\documentclass[11pt,reqno]{amsart}
\usepackage{amsmath}
\usepackage{amsthm}
\usepackage{amssymb}
\usepackage{enumerate}
\usepackage{amscd}
\usepackage{color}
\usepackage{comment}
\usepackage{tcolorbox}
\usepackage{hyperref}
\hypersetup{pdftex,colorlinks=true,allcolors=blue}
\usepackage{hypcap}
\usepackage[titletoc]{appendix}
\usepackage[alphabetic, abbrev, lite]{amsrefs}

\theoremstyle{plain}
\newtheorem{thm}{Theorem}[section]

\newtheorem{prop}[thm]{Proposition}
\newtheorem{coro}[thm]{Corollary}
\newtheorem{thm-Intro}{Theorem} 
\newtheorem{cor-Intro}{Corollary} 

\theoremstyle{definition}
\newtheorem{exs}{Examples}[section]
\newtheorem{rmk}[thm]{Remark}

\numberwithin{equation}{section}

\newcommand{\R}{\textup{R}}

\newcommand{\Hol}{\textup{Hol}}

\newcommand{\qk}{q_{_k}}

\newcommand{\qN}{q_{_N}}
\newcommand{\C}{\textup{C}}
\newcommand{\B}{\textup{B}}

\renewcommand{\Re}{\operatorname{Re}}

\begin{document}
\title[The Kobayashi-Royden metric on punctured spheres]{The Kobayashi-Royden metric on punctured spheres}

\author{Gunhee Cho, Junqing Qian}
\address{Department of Mathematics\\
University of Connecticut\\
341 Mansfield Road U1009
Storrs, Connecticut 06269-1009, U.S.}

\email{gunhee.cho@uconn.edu}
\email{junqing.qian@uconn.edu}

\begin{abstract} 
	This paper gives an explicit formula of the asymptotic expansion of the Kobayashi-Royden metric on the punctured sphere $\mathbb{CP}^1\backslash\{0,1,\infty\}$ in terms of the exponential Bell polynomials. We prove a local quantitative version of the Little Picard's theorem as an application of the asymptotic expansion. Meanwhile, the approach in the paper leads to the conclusion that the coefficients in the asymptotic expansion are rational numbers. Furthermore, the explicit metric formula and the conclusion regarding the coefficients apply to a more general case as well, the metric on $\mathbb{CP}^1\backslash\{0,\frac{1}{3},-\frac{1}{6}\pm\frac{\sqrt{3}}{6}i\}$ will be given as a concrete example of our results.
\end{abstract}

\maketitle

\tableofcontents

\section{Introduction}

In this paper, we investigate the Kobayashi-Royden metric on the one-dimensional open Riemann surface $\mathbb{CP}^1\backslash\{a_1,\ldots,a_n\}$, $n\ge 3$. People have worked on the estimation of the Kobayashi-Royden metric on punctured spheres, especially for the case $\mathbb{CP}^1\backslash\{0,1,\infty\}$ (for example, see \cite{KLZ14}). In this paper, we are able to give the explicit formula of the Kobayashi-Royden metric on $\mathbb{CP}^1\backslash\{0,1,\infty\}$ by using the exponential Bell polynomials. Further-more, the result will also apply to general cases and we will prove a local quantitative version of the Little Picard's theorem. 

In a broader point of view, given a projective manifold $X$, for any point $x\in X$, there eixsts a Zariski neighborhood $U=X\backslash Z$ of $x$ such that $U$ can be embedded into the product
$$M_1\times\cdots\times M_k$$
as a closed algebraic submanifold, where $k$ is some positive integer and each $M_r=\mathbb{CP}^1\backslash\{a_1,\ldots,a_{n_r}\}$, $n_r\ge 3$ (see \cite[p. 25, Lemma 2.3]{GZ71}). On the other side, a recent result shows that  a complete K\"ahler manifold $(M,\omega)$ with negative holomorphic sectional curvature range, the base K\"ahler metric $\omega$ is uniformly equivalent to the complete K\"ahler-Einstein metric and the Kobayashi-Royden metric \cite[Theorems 2 and 3]{WY17}. People are interested in the class of K\"ahler manifolds on which the two metrics coincide (for example, see \cites{SKY09, GC18, WY17}). We will show that the two metrics coincide on puncture spheres while other invariant metrics vanish everywhere.

A lower bound of the Kobayashi-Royden metric on $\mathbb{CP}^1\backslash\{a_1, a_2, a_3\}$ can be established by constructing a volume form (see \cite{JCPG72}). In \cite{KLZ14}, H. Kang, L. Lee and C. Zeager estimated the derivative of the modular lambda function $\lambda(\tau)$ to obtain the boundary behavior of the Kobayashi-Royden metric on $\mathbb{CP}^1\backslash\{0,1,\infty\}$. Specifically, the Kobayashi-Royden metric at the point $p$ in the direction $v$ has the following estimation
\begin{equation*}
\chi_{_{\mathbb{CP}^1\backslash\{\infty,0,1\}}}(p;v)\approx \frac{1}{\delta \log(1/\delta)}\qquad\mbox{as}\qquad \delta\to0^+,
\end{equation*}
 where $\mbox{dist}(p,0)=\delta$ and $v=1$.

In this paper, we derive a precise asymptotic expansion of the Kobayashi-Royden metric on $\mathbb{CP}^1\backslash\{a_1=0, a_2, \ldots,a_n\}$ from the asymptotic expansion of the complete K\"ahler-Einstein metric in \cite{JQ19}. Futhermore, we enhance the formula in \cite{JQ19} by using the exponential Bell polynomials. As an application, we also prove a local quantitative version of the Little Picard's theorem.

Let us denote the Kobayashi-Royden metric on a complex manifold $N$ at $p\in N$ in the direction $v\in T_pN$ by $\chi_{_N}(p;v)$, and sometime we will write $M=\mathbb{CP}^1\backslash\{a_1=0,a_2,\ldots,a_n\}$ for convenience. Let $f:\mathbb{H}\rightarrow\mathbb{CP}^1\backslash\{a_1=0,a_2,\ldots,a_n\}$ be a covering map with $f(\infty)=0$, then we have the following expansion of $f$
\begin{equation}\label{eq:intro-cover-def}
	f(\qk)=\sum_{m=1}^{\infty}c_m\qk^m,
\end{equation}
 and the inversion series $\qk=\qk(f)$ around $f=0$,
\begin{equation}\label{eq:intro-inversion}
	\qk(f)=\sum_{m=1}^{\infty}b_mf^m,
\end{equation}
where $\qk=\exp\{\frac{2\pi i}{k}\tau\}$, $\tau\in\mathbb{H}$. The exponential Bell polynomial is defined as the following
\begin{equation}
\B_{n,k}(t_1,t_2,\ldots,t_{n-k+1})=\sum\frac{n!t_1^{r_1}\cdots t_{n-k+1}^{r_{n-k+1}}}{r_1!\cdots r_{n-k+1}!(1!)^{r_1}\cdots ((n-k+1)!)^{r_{n-k+1}}},\notag
\end{equation}
the summation takes place over all integers $r_1,\ldots,r_{n-k+1}\ge 0$ such that
$$\begin{aligned}
&r_1+r_2+\cdots+r_{n-k+1}=k,\\
&r_1+2r_2+\cdots+(n-k+1)r_{n-k+1}=n.
\end{aligned}$$
The followings are the results of this paper.
\begin{thm}\label{thm:intro-main}
	For a covering space  $\mathbb{H}\rightarrow\mathbb{CP}^1\backslash\{a_1=0, a_2,\ldots,a_n\}$, let $f(\tau)$ be the universal covering map given by equation \eqref{eq:intro-cover-def}. Then Kobayashi-Royden metric at point $p\in\mathbb{CP}^1\backslash\{a_1=0, a_2, \ldots,a_n\}$ near the boundary $0$ has the following asymptotic expansion
	\begin{align}
	\chi_{_M}(p;v)=\frac{1}{|p||\log|b_1p||}\cdot\left|1+\sum_{m=1}^{\infty}\left(\sum_{k=1}^{m}\frac{l_k\C_{m-k}(p) p^k}{(k-1)!(m-k)!}+\frac{\C_m(p)}{m!}\right)\right|\cdot||v||,
	\end{align}
	  where 
	\begin{align}
	\C_0(p)&=0, & \notag\\
	\C_m(p)&=\sum_{k=1}^{m}\frac{(-1)^kk!}{\log^k|b_1p|}\B_{m,k}(\Re(l_1p),\ldots,\Re(l_{m-k+1}p^{m-k+1})),& \mbox{for }m\ge 1,\notag\\
	l_m&=\sum_{k=1}^{m}\frac{(-1)^{k-1}(k-1)!}{b_1^k}\B_{m,k}(1!b_2, 2!b_3, \ldots, (m-k+1)!b_{n-k+2}), &\mbox{for }m\ge 1,\notag
	\end{align}
	and $b_1, b_2, \ldots,$ are the coefficients in equation \eqref{eq:intro-inversion}.
\end{thm}

Consequently, we show the the following local quantitative version of the Little Picard's theorem as an application of Theorem \ref{thm:intro-main}.
\begin{thm}
	 Let $R> 0$ be the maximal radius of the existence of a holomorphic map $\varphi : \mathbb{D}_R \rightarrow \mathbb{CP}^1\backslash\{0,1,\infty\}$ satisfying $\varphi(0)=p$ and $\varphi'(0)=1$,  i.e.,
	 $$R=\sup\{R_0:\varphi\in\Hol(D_{R_0},\mathbb{CP}^1\backslash\{0,1,\infty\}), \varphi(0)=p, \varphi'(0)=1\},$$
	 where $p\in\mathbb{CP}^1\backslash\{0,1,\infty\}$ is a point near $0$, and $\mathbb{D}_R$ is the disk $\{z\in\mathbb{C}: |z|<R\}$. Then the upper bound of the maximal radius $R$ can be given by the following
	 \begin{equation}
	 R< \left|p\log|p/16|+\frac{1}{2}p\Re(p)-p^2\log|p/16|+O(|p|^3)\right|,\qquad\mbox{as }p\rightarrow 0.
	 \end{equation}
	 Specially, such $R\rightarrow 0$ when $p\rightarrow 0$.
\end{thm}

\section{The Kobayashi-Royden metric on punctured spheres} 

Let $N$ be a complex manifold and let $\mathbb{D}$ be the unit disk in $\mathbb{C}$. For any point $z\in N$ and a tangent vector $v\in T_{z}N$, the Kobayashi-Royden metric is defined as
\begin{equation}\label{eq:def-kobayashhi}
\chi_{_N}(z;v)=\inf\{\frac{1}{\alpha} : \alpha>0, f\in \Hol(\mathbb{D},N), f(0)=z, f'(0)=\alpha v\}.
\end{equation}

The Kobayashi-Royden metric $\chi_{_N}$ is one of the effective invariant metrics to the study of holomorphic maps; however, it is not easy to compute for arbitrary complex manifolds (see \cite{W93}). One natural way to compute $\chi_{_N}$ is to consider the case when $N$ admits a holomorphic covering $\pi : \tilde{N} \rightarrow N$, where $\tilde{N}$ is a complex manifold. Actually, the holomorphic covering $\pi$ becomes an isometry between $\tilde{N}$ and $N$ with respect to the Kobayashi-Royden metric due to the standard lifting argument of the covering map $\pi$ (see Theorem 7.3.1 in \cite{GKK11} and Exercise 3.9.8 in \cite{JP13}). We have the following theorem showing the relation between the Kobayashi-Royden metric and the K\"ahler-Einstein metric.

\begin{prop}\label{thm}
	Let $\pi:\tilde{N}\rightarrow N$ be a holomorphic covering between two complex manifolds $\tilde{N}$ and $N$. Assume $\tilde{N}$ possesses the K\"ahler-Einstein metric $\omega_{_{\tilde{N}}}$ of negative Ricci curvature, the Kobayashi-Royden metric $\chi_{_{\tilde{N}}}$, and these two metrics coincide in the following sense $$\sqrt{\omega_{_{\tilde{N}}}(v,v)}=\chi_{_{\tilde{N}}}(p;v),\quad \mbox{for any }v\in T_{\tilde{p}}\tilde{N},\, \tilde{p}\in \tilde{N}.$$ 
	Then $N$ also possesses the K\"ahler-Einstein metric $\omega_{_N}$ and the Kobayashi-Royden metric $\chi_{_N}$. Furthermore, the two metrics coincide on $N$ as well, and they satisfy the following relation
\begin{equation*}
\sqrt{\omega_{_N}(v,v)}=\chi_{_N}(p;v)=\chi_{_{\tilde{N}}}(\tilde{p};v)\frac{||v||}{|\pi'(\tilde{p})|},
\end{equation*}	
where $v\in T_{p}N$ and $\pi(\tilde{p})=p$.
\end{prop}
\begin{proof}
	For $\dim_\mathbb{C}N\ge 2 $, it follows directly from \cite[p.207, Corollary 7.6.4]{GKK11}. For the case of $\dim_\mathbb{C}N=1$, it follows from \cite[p.122, Theorem 3.3.7]{JP13} (see, also \cite[Lemma 1]{KLZ14}).
\end{proof}

Let $f$ be the covering map from $\mathbb{H}$ to $\mathbb{CP}^1\backslash\{a_1=0, a_2, \ldots,a_n\}$ with $f(\infty)=0$, then $f$ can be written as the following expansion
\begin{equation}\label{eq:Cm-coefficients}
f(\tau)=f(\qk)=c_1\qk+c_2\qk^2+\sum_{m=3}^{\infty}c_m\qk^m,
\end{equation}
where $\qk=\exp\{\frac{2\pi i}{k}\tau\}$, $\tau\in\mathbb{H}$, $k$ is a real constant and it holds for any $\tau\in\mathbb{H}$ (\cite[Proposition 3.10]{JQ19}). Let us write $\qk$ as a series in $f$, 
\begin{equation}\label{eq:Dm-coefficients}
\qk(f)=b_1f+b_2f^2+\sum_{m=3}^{\infty}b_mf^m,
\end{equation}
where $b_1=1/c_1$ and 
\begin{equation}\label{eq:D_m-explicit}
	b_m=\frac{1}{m!}\sum_{k=1}^{m-1}\frac{(-1)^k}{c_1^{m+k}}\B_{m+k-1,k}(0, 2!c_2, 3!c_3, \ldots, m!c_m),\qquad m\ge 2,
\end{equation}
where $\B_{n,k}$ is the exponential Bell polynomial (see \cite[p. 134 and p. 151]{CL-Comb} and \cite{BE34Ann}) defined as the following
\begin{equation}\label{bell-poly-def}
	\B_{n,k}(t_1,t_2,\ldots,t_{n-k+1})=\sum\frac{n!t_1^{r_1}\cdots t_{n-k+1}^{r_{n-k+1}}}{r_1!\cdots r_{n-k+1}!(1!)^{r_1}\cdots ((n-k+1)!)^{r_{n-k+1}}},
\end{equation}
the summation takes place over all integers $r_1,\ldots,r_{n-k+1}\ge 0$ such that
$$\begin{aligned}
&r_1+r_2+\cdots+r_{n-k+1}=k,\\
&r_1+2r_2+\cdots+(n-k+1)r_{n-k+1}=n.
\end{aligned}$$

\begin{thm}
For a covering space  $\mathbb{H}\rightarrow\mathbb{CP}^1\backslash\{a_1=0, a_2, \ldots,a_n\}$, let $f(\tau)$ be the universal covering map given by equation \eqref{eq:Cm-coefficients}. Then the Kobayashi-Royden metric at $p\in\mathbb{CP}^1\backslash\{a_1=0, a_2, \ldots,a_n\}$ near $0$ is given as the following
\begin{align}\label{eq:metric-formula}
	\chi_{_M}(p;v)=\frac{1}{|p||\log|b_1p||}\cdot\left|1+\sum_{m=1}^{\infty}\left(\sum_{k=1}^{m}\frac{l_k\C_{m-k}(p) p^k}{(k-1)!(m-k)!}+\frac{\C_m(p)}{m!}\right)\right|\cdot||v||,
\end{align}
where
\begin{align}
	\C_m(p)&=\sum_{k=1}^{m}\frac{(-1)^kk!}{\log^k|b_1p|}\B_{m,k}(\Re(l_1p), \ldots, \Re(l_{m-k+1}p^{m-k+1})),\quad \C_0(p)=0,\notag\\
	l_m&=\sum_{k=1}^{m}\frac{(-1)^{k-1}(k-1)!}{b_1^k}\B_{m,k}(1!b_2, 2!b_3, \ldots, (m-k+1)!b_{m-k+2}).\notag
\end{align}
\end{thm}

\begin{proof}
	From Theorem \ref{thm} and \cite[Theorem 3.11]{JQ19}, we have
	\begin{align}\label{eq:metric-definition-induced}
		\chi_{_M}(p;v)=|ds|=\frac{|q'(p)|}{|q(p)|\log|q(p)|}||v||.
	\end{align}
	Let us consider $q_{_f}/q$ as $(\log q)_{_f}$, which is well-defined due to the differentiation. Let us fix a branch and use equation \eqref{eq:Dm-coefficients}, we have the following logarithm expansion (\cite[p. 141]{CL-Comb}) 
	\begin{align}
		\log q&=\log(b_1f)+\log\left(1+\sum_{m=1}^{\infty}\frac{b_{m+1}}{b_1}f^m\right)\notag\\
		         &=\log(b_1f)+\sum_{m=1}^{\infty}\frac{l_m}{m!}f^m,\label{eq-log-q}
	\end{align}
	where
	\begin{equation}\label{eq:def-L-m}
l_m=\sum_{k=1}^{m}\frac{(-1)^{k-1}(k-1)!}{b_1^k}\B_{m,k}(1!b_2, 2!b_3, \ldots, (m-k+1)!b_{m-k+2}).
	\end{equation}
Therefore we have
	\begin{equation}
		(\log q(f))'=\frac{1}{f}+\sum_{m=1}^{\infty}\frac{l_m}{(m-1)!}f^{m-1}=\frac{1}{f}+\frac{b_2}{b_1}+\sum_{m=1}^{\infty}\frac{l_{m+1}}{m!}f^m.\notag
	\end{equation}
On the other side, we have
$$\log|1+z|=\frac{1}{2}\log|1+z|^2=\frac{1}{2}\log(1+z)(1+\overline{z}),$$
let us consider both of $\log(1+z)$ and $\log(1+\overline{z})$ on the principal branch,
\begin{align}
2\log|1+z|&=\log(1+z)+\log(1+\overline{z})\notag\\
	&=\sum_{m=1}^{\infty}(-1)^{m+1}\frac{z^m}{m}+\sum_{m=1}^{\infty}(-1)^{m+1}\frac{\overline{z}^m}{m}\notag\\
	&=2\sum_{m=1}^{\infty}(-1)^{m+1}\frac{\Re(z^m)}{m}.\notag
\end{align}
The asymptotic expansion of the metric can be given as the following
	\begin{align}
		\frac{|q'|}{|q||\log|q||}&=|(\log(q))'|\cdot|\log|q||^{-1}\notag\\
		      &=|(\log(q))'|\cdot |\log|b_1f||^{-1}\cdot\left|1+\sum_{m=1}^{\infty}\frac{1}{m!\log|b_1f|}\Re\left(l_mf^m\right)\right|^{-1}\notag\\
		      &=\left|\frac{(\log q)'}{\log|b_1f|}\right|\cdot\left|1+\sum_{m=1}^{\infty}\frac{1}{m!}\C_m(f)\right|\notag\\
		      &=\frac{1}{|f||\log|b_1f||}\cdot\left|1+\sum_{m=1}^{\infty}\frac{ml_m}{m!}f^m\right|\cdot\left|1+\sum_{m=1}^{\infty}\frac{C_m(f)}{m!}\right|\notag\\
		      &=\frac{1}{|f||\log|b_1f||}\cdot\left|1+\sum_{m=1}^{\infty}\R_m(f)\right|,
	\end{align}
	where $\C_m(f)$ and $\R_m(f)$ are defined as
	\begin{align}
		\C_m(f)&=\sum_{k=1}^{m}\frac{(-1)^kk!}{\log^k|b_1f|}\B_{m,k}(\Re(l_1f), \ldots, \Re(l_{m-k+1}f^{m-k+1})),\quad \C_0(f)=0,\notag\\
		\R_m(f)&=\sum_{k=1}^{m}\frac{l_k\C_{m-k}(f)f^k}{(k-1)!(m-k)!}+\frac{\C_m(f)}{m!}.\label{eq-R-m}
	\end{align}
\end{proof}

For some special cases, we are able to obtain more information on the coefficients of the metric expansion.

\begin{thm}\label{thm2}
	The Kobayashi-Royden metric on $\mathbb{CP}^1\backslash\{a_1=0, a_2, \ldots,a_n\}$, $n=3, 4, 6, 12$, has the following asymptotic expansion
\begin{align}
\chi_{_M}(p;v)=\frac{1}{|p||\log|b_1p||}\cdot\left|1+\sum_{m=1}^{\infty}\R_m(p)\right|\cdot||v||,
\end{align}
where $\R_m$ is defined by equation \eqref{eq-R-m}, and the coefficients $l_m$ in $\R_m(p)$ satisfy that $l_m\in \mathbb{Q}(c_1, c_2)$.
\end{thm}
\begin{proof}
Notice that we have the following relation
\begin{align}\label{eq:E4-Swarzian}
1+240\sum_{m=1}^{\infty}\sigma_3(m)q^m=1-q^2_{_N}\{f,q_{_N}\}, \quad N=2,3,4,5,
\end{align}
where $\sigma_3(m)=\sum_{d|m}d^3$ (\cite[Theorem 8.2]{JQ19}), and $\{f,q_{_N}\}$ is the Schwarzian derivative defined as the following 
\begin{align}\label{eq:Sch}
\{f,q_{_N}\}&=2\left(\frac{f_{\qN\qN}}{f_{\qN}}\right)_{\qN}-\left(\frac{f_{\qN\qN}}{f_{\qN}}\right)^2.
\end{align}
We write $q=q_{_N}$ for convenience. Let us calculate $f_{qq}/f_{q}$ by finding the series expansion of $(\log f_{q})_{q}$, the branch choice is no longer an issue due to the differentiation,
\begin{align}
	(\log(f_{q}))_{q}&=\left[\log\left(c_1+\sum_{m=1}^{\infty}\frac{(m+1)!c_{m+1}}{m!}q^m\right)\right]_q\notag\\
	                    &=\left(\log(c_1)+\sum_{m=1}^{\infty}\frac{\tilde{l}_m}{m!}q^m\right)_q\notag\\
	                    &=\tilde{l}_1+\sum_{m=1}^{\infty}\frac{\tilde{l}_{m+1}}{m!}q^m,\notag
\end{align}
where 
\begin{equation}
	\tilde{l}_m=\sum_{k=1}^{m}\frac{(-1)^{k-1}(k-1)!}{c^k_1}\B_{m,k}(2!c_2, \ldots, (m-k+2)!c_{m-k+2}).
\end{equation}
We have the following series expansion
\begin{equation}
	\{f,q\}=\sum_{m=0}^{\infty}\left(\frac{2\tilde{l}_{m+2}}{m!}-\sum_{k=0}^{m}\frac{\tilde{l}_{k+1}\tilde{l}_{m-k+1}}{k!(m-k)!}\right)q^m
\end{equation}
for the Schwarzian derivative. We can solve for each $c_m$ from the following relation
\begin{equation}
	1+240\sum_{m=1}^{\infty}\sigma_3(m)\qN^m=1-\sum_{m=0}^{\infty}\left(\frac{2\tilde{l}_{m+2}}{m!}-\sum_{k=0}^{m}\frac{\tilde{l}_{k+1}\tilde{l}_{m-k+1}}{k!(m-k)!}\right)\qN^{m+2}.
\end{equation}
Specifically, for $m\ge 0$, we have equations
\begin{equation}\label{hugerelation}
\left\{\begin{aligned}
&\frac{2\tilde{l}_{m+2}}{m!}-\sum_{k=0}^{m}\frac{\tilde{l}_{k+1}\tilde{l}_{m-k+1}}{k!(m-k)!}=-240\sigma_3(\frac{m+2}{N}), &\mbox{if }N|m+2,\\
&\frac{2\tilde{l}_{m+2}}{m!}-\sum_{k=0}^{m}\frac{\tilde{l}_{k+1}\tilde{l}_{m-k+1}}{k!(m-k)!}=0,&\mbox{otherwise}.
\end{aligned}\right.
\end{equation}
 Each $c_{m+2}$ can be solved from $\tilde{l}_{m+2}$, and $\tilde{l}_{m+2}$ can be solved in terms of $\tilde{l}_1,\ldots,\tilde{l}_{m+1}$. Relation \eqref{hugerelation} implies that $c_{m+2}\in\mathbb{Q}(c_1, c_2)$, consequently, $b_m\in\mathbb{Q}(c_1, c_2)$ due to equation \eqref{eq:D_m-explicit}, so as $l_m\in\mathbb{Q}(c_1, c_2)$ due to equation \eqref{eq:def-L-m}.
\end{proof}

\begin{coro}\label{thm:KR}
The Kobayashi-Royden metric on $\mathbb{CP}^1\backslash\{0,1,\infty\}$ has the following asymptotic expansion
\begin{equation}
\chi_{_M}(p,v)=\frac{1}{|p|\log|\frac{p}{16}|}\left\{1+\frac{1}{2}\left(p-\frac{\Re p}{\log|\frac{p}{16}|}\right)+O(|p|^2)\right\}||v||,\qquad\mbox{as } p\rightarrow 0.\nonumber
\end{equation}

\end{coro}
\begin{proof}
It directly follows from Theorem \ref{thm2}, and the fact that $c_1=16$, $c_2=-128$ in equation \eqref{eq:Cm-coefficients}.
\end{proof}

\begin{rmk}\label{rmk2.3}
For $M=\mathbb{CP}^1\backslash\{a_1=0,a_2,\ldots,a_n=\infty\}$, $n\geq 3$, we can identify $M$ with $\mathbb{C}^1\backslash\{a_1=0,a_2, \ldots,a_{n-1}\}$. Hence $M$ is a quasi-projective algebraic manifold which is biholomorphic to 
	\begin{equation*}
	\{(x,y)\in \mathbb{C}^2 : y(x-a_1)\cdots(x-a_{n-1})-1=0 \}.
	\end{equation*}
Therefore our result is a concrete example of the Kobayashi-Royden pseudometric on a quasi-projective algebraic manifold $Z$ which satisfies
	\begin{equation*}
	\chi_{_Z}(p;v)=\inf_{_C}\chi_{_C}(p;v),
	\end{equation*}
	where $C$ runs over all (possibly singular) closed algebraic curves in $Z$ that is tangent to $v$ (see Corollary 1.3 in \cite{DJPLLSB94}). 
\end{rmk}


\section{An application of the Kobayashi-Royden metric on $\mathbb{CP}^1\backslash\{0,1,\infty\}$}

The Little Picard's theorem is one of the classic theorems in complex analysis states that there is no non-trivial holomorphic map from $\mathbb{C}$ to $\mathbb{CP}^1\backslash\{0,1,\infty\}$. Hence it is natural to ask the maximal radius $R>0$ of the existence of a non-trivial holomorphic function $\varphi : {\mathbb{D}_R} \rightarrow \mathbb{CP}^1\backslash\{0,1,\infty\}$, where $\mathbb{D}_R=\{z\in \mathbb{C} : |z|<R \}$. For this question, the Kobayashi-Royden metric gives a local quantitative information that can be used to determine $R$. Recall the definition of the Kobayashi-Royden metric at a point $p\in N$ in the direction of $v\in T_pN$ in equation \eqref{eq:def-kobayashhi},
$$\chi_{_N}(p;v)=\inf\{\frac{1}{\alpha} : \alpha>0, f\in \Hol(\mathbb{D},N), f(0)=p, f'(0)=\alpha v\}.$$
Let us consider the collection of holomorphic maps defined on $\mathbb{D}_R$ instead of defined on the unit disk $\mathbb{D}$, and take the tangent vector $v=1$, we have the following alternative definition
\begin{equation}\label{eq:alternative-R}
	\chi_{_N}(p;1)=\inf\{\frac{1}{R}: \alpha >0, g\in\Hol(\mathbb{D}_R,N), g(0)=0, g'(0)=1\}.
\end{equation}
We use the asymptotic expansion of the Kobayashi-Royden metric on ${\mathbb{CP}^1\backslash\{\infty,0,1\}}$ to obtain the following local quantitative version of the Little Picard's theorem.

\begin{thm}
Let $R> 0$ be the maximal radius of the existence of a holomorphic map $\varphi : {\mathbb{D}_R} \rightarrow \mathbb{CP}^1\backslash\{0,1,\infty\}$ satisfying $\varphi(0)=p$ and $\varphi'(0)=1$,  i.e.,
$$R=\sup\{R_0:\varphi\in\Hol(\mathbb{D}_{R_0},\mathbb{CP}^1\backslash\{0,1,\infty\}), \varphi(0)=p, \varphi'(0)=1\},$$
where $p\in\mathbb{CP}^1\backslash\{0,1,\infty\}$ is a point near $0$. Then the upper bound of the maximal radius can be given by the following
\begin{equation}\label{eq:rigidity}
R< \left|p\log|p/16|+\frac{1}{2}p\Re(p)-p^2\log|p/16|+O(|p|^3)\right|,\qquad\mbox{as }p\rightarrow 0.
\end{equation}
Specially, such $R\rightarrow 0$ when $p\rightarrow 0$.
\end{thm}
\begin{proof}
Let us take $v=1$, equations \eqref{eq:metric-definition-induced} and \eqref{eq:alternative-R} imply the following equation
$$\chi_{_{\mathbb{CP}^1\backslash\{\infty,0,1\}}}(p;1)=\inf \frac{1}{R_0}=R.$$
More precisely, the upper bound of the radius can be approximated as the following
\begin{align}
R&<\frac{1}{\chi_{_{\mathbb{CP}^1\backslash\{\infty,0,1\}}}(p;1)}\\
&=\frac{|\log|q||}{|(\log q)'|}\notag\\
&=|\log|q||\cdot|f|\cdot\left|1+\sum_{m=1}^{\infty}\frac{ml_m}{m!}f^m\right|^{-1}\notag\\
&=|\log|q||\cdot|f|\cdot\left|1+\sum_{m=1}^{\infty}\frac{\tilde{c}_m}{m!}f^m\right|\notag\\
&=|f||\log|b_1f||\cdot\left|1+\sum_{m=1}^{\infty}\frac{1}{m!\log|b_1f|}\Re\left(l_mf^m\right)\right|\cdot \left|1+\sum_{m=1}^{\infty}\frac{\tilde{c}_m}{m!}f^m\right|\notag\\
&=|f||\log|b_1f||\cdot\left|1+\sum_{m=1}^{\infty}\left(\frac{\tilde{c}_mf^m}{m!}+\sum_{k=1}^{m}\frac{\tilde{c}_{m-k}f^{m-k}\Re(l_kf^k)}{(m-k)!k!\log|b_1f|}\right)\right|,
\end{align}
where
\begin{align}
	\tilde{c}_m&=\sum_{k=1}^{m}(-1)^kk!\B_{m,k}(l_1, 2l_2, \ldots, (m-k+1)l_{m-k+1}),\notag\\
	l_m&=\sum_{k=1}^{m}\frac{(-1)^{k-1}(k-1)!}{b_1^k}\B_{m,k}(1!b_2, 2!b_3, \ldots, (m-k+1)!b_{m-k+2}).\notag
\end{align}
Notice that $b_1=\frac{1}{16}$ and $b_2=\frac{1}{32}$. For $m=1$, we have
$$\frac{\tilde{c}_1f^1}{1!}+\sum_{k=1}^{1}\frac{\tilde{c}_{1-k}f^{1-k}\Re(l_kf^k)}{(1-k)!k!\log|f/16|}=\frac{1}{2}\left(f\Re f-f^2\log|f/16|\right).$$
For $m\ge 2$, we have
$$\frac{\tilde{c}_mf^m}{m!}+\sum_{k=1}^{m}\frac{\tilde{c}_{m-k}f^{m-k}\Re(l_kf^k)}{(m-k)!k!\log|f/16|}=O(f^2)\qquad\mbox{as}\qquad f\rightarrow 0.$$
The conclusion follows directly. 
\end{proof}
\begin{rmk}
	The estimation of the maximal radius $R$ is sharp when the point $p$ is close to $0$. It does not work for a point $p$ if $|p|$ is large, because the inversion series \eqref{eq:Dm-coefficients} converges locally, thus the asymptotic expansion of the Kobayashi-Royden metric works locally.
\end{rmk}

\section{Example}
\begin{exs}
	For the punctured sphere $\mathbb{CP}^1\backslash\{0,\frac{1}{3}, -\frac{1}{6}\pm\frac{\sqrt{3}}{6}i\}$, a universal covering map can be given by
	$$f(\tau)=\left(\frac{\eta(3\tau)}{\eta(\frac{\tau}{3})}\right)^3=q_{_3}+3q^2_{_3}+\sum_{m=3}^{\infty}c_mq_{_3}^m,$$
	and the corresponding deck transformation group is $\Gamma(3)$ (see \cite{SA02}), where $\eta(\tau)=q^{\frac{1}{24}}\prod_{n=1}^{\infty}(1-q^n)$ is the Dedekind eta function, here $q=\exp\{2\pi i\tau\}$ and $\tau\in\mathbb{H}$. In this case, we have 
	$$c_1=1 \quad\mbox{and}\quad c_2=3.$$
    The first couple of coefficients in $q=\sum_{m=1}^{\infty}b_mf$ can be solved as the following by Theorems \ref{thm} and \ref{thm2}, 
		$$c_3=9, b_1=1,\,b_2=-3,\,b_3=9,$$
		so we have
		$$\begin{aligned}
		\chi_{_M}(p;v)&=\frac{1}{|p|\log|p|}\left|1-3\left(p-\frac{\Re p}{\log |p|}\right)\right.\\
		&\qquad\qquad\qquad\left.+9\left(p^2-\frac{p\Re p}{\log|p|}-\frac{1}{2}\frac{\Re(p^2)}{\log|p|}+\frac{(\Re p)^2}{\log^2|p|}\right)+O(|p|^3)\right|||v||,
		\end{aligned}
		$$
		as $p\rightarrow 0$, where $M=\mathbb{CP}^1\backslash\{0,\frac{1}{3}, -\frac{1}{6}\pm\frac{\sqrt{3}}{6}i\}$, $p\in M$ and $v\in T_pM$.
\end{exs}

\subsection*{Acknowledgements}
This work was partially supported by NSF grant DMS-1611745. We would like to thank our advisor Professor Damin Wu for helpful discussions and the encouragement of making the collaboration. We would also like to thank Professor Keith Conrad for useful comments.





\end{document}